\documentclass[12pt]{amsart}
\usepackage{amsfonts,amsmath,amssymb,amsthm, geometry}
\usepackage[T1]{fontenc}
\usepackage[latin9]{inputenc}
\usepackage[all]{xy}
\usepackage[active]{srcltx}
\usepackage{graphicx} 
\usepackage{amsfonts,color}

\makeatletter

\newtheorem{theorem}{Theorem}[section]
\newtheorem{corollary}[theorem]{Corollary}
\newtheorem{lemma}[theorem]{Lemma}

\theoremstyle{definition}
\newtheorem{definition}[theorem]{Definition}
\theoremstyle{remark}
\newtheorem{remark}[theorem]{\sc Remark}
\newtheorem{example}[theorem]{\sc Example}

\renewcommand{\int}{{\mathrm{int}}}

\newcommand{\Disc}{\mathop{{\mathrm{Disc}}}\nolimits}

\newcommand{\Sing}{\mathrm{Sing\hspace{1pt}}}

\newcommand{\grad}{\mathop{\rm{grad}}\nolimits}

\renewcommand{\d}{{\mathrm{d}}}

\newcommand{\e}{\varepsilon}

\newcommand{\m}{\setminus}

\newcommand{\im}{{\mathrm{Im}}}
\newcommand{\rank}{{\mathrm{rank\hspace{2pt}}}}

\newcommand{\bR}{{\mathbb R}}
\newcommand{\bC}{{\mathbb C}}

\newcommand{\bz}{{\mathbf{z}}}

\begin{document}

\title[Thom irregularity  and Milnor tube fibrations]{Thom irregularity and Milnor tube fibrations}


\author{A.J. Parameswaran}

\address{School of Mathematics, Tata Institute of Fundamental
Research, Homi Bhabha Road, Mumbai 400005, India}
\email{param@math.tifr.res.in}

\author{Mihai Tib\u ar}
\address{Univ. Lille, CNRS, UMR 8524 -  Laboratoire Paul Painlev\'e,  F-59000 Lille, France}
\email{mihai-marius.tibar@univ-lille.fr}

\subjclass[2000]{14D06, 14P15,  32S20, 58K05, 57R45, 58K15}

\keywords{singularities of real polynomial maps, fibrations, bifurcation locus}

\thanks{The authors acknowledge the support of the Labex CEMPI (ANR-11-LABX-0007-01) at Lille and 
of the Mathematisches Forschungsinstitut  Oberwolfach through a RiP program.}

\begin{abstract}

We find natural and convenient conditions which allow us to produce  classes of genuine real map germs with Milnor tube fibration, either  with Thom regularity or without  it.

\end{abstract}
\maketitle


\section{Introduction}\label{s:intro}

The existence of fibered structures is fundamental in the study of spaces and maps. For holomorphic function germs $f : (\bC^n,0)\to (\bC, 0)$, $n >1$, it is well known that around the origin there is a  ``Milnor tube fibration'' (see Definition \ref{d:tube}). This has been discovered by Milnor \cite{Mi} in case of isolated singularities and later by Hironaka \cite{Hi} and L\^{e} D.T. \cite{Le}  for any singular locus and was the starting point of a tremendous work ever since.

The existence of the Milnor  fibration is not insured anymore  in case of maps.
Holomorphic  map germs $G: (\bC^n,0)\to (\bC^p, 0)$, $p>1$, have been studied in the 80's in the fundamental work by Sabbah \cite{Sa1}, \cite{Sa2} and Henry, Merle, Sabbah \cite{HMS} (treating a more general setting)
in connexion with the following problem posed by L\^e D.T. and issued from questions raised by Thom and Deligne: \emph{give a criterion for $G$ to have a Milnor fibration in sufficiently small ball neighbourhoods}.   The class of holomorphic map germs called \emph{ICIS} (isolated complete intersection singularity) has a Milnor fibration outside the discriminant $\Disc G$ due to the isolatedness of the singular locus of the central fibre $G^{-1}(0)$. Its topology has been studied by different methods in numerous papers ever since \cite{Ha, Gr, GH, Eb, ES, Pw1, Pw2} etc.

Hamm and L\^{e}  \cite{HL, Le} have proved  that the existence of a Thom regular stratification of some complex map insures 
the existence of vanishing cycles at any point, i.e. the existence of a Milnor fibration outside the discriminant.  The existence of a Thom regular stratification has been characterised in \cite{HMS} by the property called ``no blowing-up in codimension 0'' which can be described in terms of polar invariants. Integral closure of modules has been used by Gaffney \cite{Ga} in his study of the Thom regularity.

\medskip

 
In this paper we consider the real setting.  We work with the following definition:
 \begin{definition}\label{d:tube}
Let $G:(\mathbb{R}^{m},0) \rightarrow (\mathbb{R}^{p}, 0)$, $m\ge p\ge 1$, be a non-constant map germ.
We say that $G$ has {\em Milnor tube fibration} if $\Sing G\subset G^{-1}(0)$  and if, for any $\e > 0$ small enough, there exists some positive $\eta \ll \e$ such that the restriction:
\begin{equation}\label{eq:tube}
 G_| : \bar B^{m}_{\e} \cap G^{-1}( B^{p}_\eta \m \{0\}) \to  B^{p}_\eta \m \{0\} 
\end{equation}
 is a \emph{locally trivial fibration}, the topological type of which is \emph{independent} of the choices of $\e$ and $\eta$.
 \end{definition}

The existence of Milnor tube fibrations is implied\footnote{but is not equivalent, see \S 5.}  by  the existence of a Thom (a$_{G}$)-regular stratification at $G^{-1}(0)$,  exactly like in the complex holomorphic setting.  We shall say, for short,  that  \emph{$G$ is Thom regular}, or that \emph{$G$ has Thom regularity.}

The first strategy would be to find natural and convenient enough conditions for the Thom regularity  such that to produce large 
classes of  maps $G$ with Thom regularity and therefore with Milnor tube fibration.
A sufficient condition of \L ojasiewicz type has been formulated by Massey's  \cite{Ma}; however it  did not yield new examples outside the class of functions (real or complex  analytic) where the classical \L ojasiewicz inequality holds and implies the Thom regularity, cf \cite{HL}. More efficient in  producing new examples seem to be some results in particular settings which can be checked more easily, see \S \ref{proof} for  a short discussion:  in case of mixed functions $G : \bC^{n}\to \bC$ which are Newton strongly non-degenerate \cite{Oka}; in case of  certain functions of type  $f\bar g$ in 2 variables \cite{PS} and  in 3 variables \cite{FM}, where $f$ and $g$ are holomorphic germs.
 
 \smallskip

 Along the first strategy, we focus here to the class of mixed functions of type $f\bar g$  in any number $n>1$ of coordinates, and therefore to the  following two problems:

\smallskip

\noindent
(i)  characterize in simple terms the condition  ``$f\bar g$ has isolated critical value''.

\noindent
(ii)  give natural and sharp conditions under which $f\bar g$ has a Thom regular stratification.

\smallskip

Our results are as follows. 

Theorem \ref{t:main1} proves that
 \emph{$f\bar g :  (\bC^{n}, 0) \to (\bC, 0)$ has an isolated critical value if and only if the discriminant $\Disc (f,g)$ is tangent to the coordinate axes.} 
 
Next, our Theorem \ref{t:mainthom1} proves that \emph{the Thom irregularity locus of $f\bar g$ is included into that of the holomorphic map $(f,g)$.}
  Consequently, if $(f,g)$ has a Thom stratification, e.g. if it is an ICIS,   then $f\bar g$ is Thom regular (and has a highly non-isolated singularity for $n>2$).
From this handy result we draw significant classes of maps $f\bar g$ which  have  isolated critical value, nonisolated singularities, and are Thom regular  (cf \S \ref{MilnorwithThom}).
 Another such class is provided by our Corollary \ref{p:thomseb},  which is a Thom-Sebastiani type result: \emph{for $f$ and $g$ in separate variables,  $f\bar g$ is Thom regular.}

\medskip

The second
 strategy would be to produce  classes
 of real maps with Milnor tube fibrations but \emph{without Thom regularity}.
The possibility of this  phenomenon has been indicated by the first examples  of real maps with \emph{singular open book structure} and  without Thom regularity,  given in \cite{Ti-obw}, \cite[\S 5.3]{ACT}, \cite{Pa}. 

In this vein, by Theorem \ref{t:homogen} we provide a sufficient condition in $n>1$ variables: \emph{if $F$ is a non-constant mixed germ which is polar weighted-homogeneous, then $F$  has Milnor tube fibration.}  This enables us to show, in particular,  that the functions of the infinite class
 $F(x,y,z)= (x+z^k)y\bar x$, $k\ge 2$, have noninsolated singularities, have Milnor tube fibration, but do not have Thom regularity.


\section{Mixed functions}\label{mixed}

The terminology \emph{mixed polynomial} has been introduced by Mutsuo Oka  \cite{Oka1, Oka2, Oka3} and refers to the following setting.
Let $F:= (u,v) : \bR^{2n} \to \bR^2$ be a polynomial map germ at the origin and let us fix some coordinates  $(x_1, y_1,  \ldots , x_n, y_n)$. By writing $\mathbf{z}=\mathbf{x}+i\mathbf{y}\in \bC^n$, where $z_{k}=x_{k}+iy_{k} \in \bC$ for $k=1,\ldots,n$, 
one gets a polynomial function $F: \bC^n \to \bC$ in variables $\bz$ and $\bar \bz$, namely 
$F(\mathbf{z},\mathbf{\bar{z}}):=u(\frac{\mathbf{z}+\mathbf{\bar{z}}}{2},\frac{\mathbf{z}-\mathbf{\bar{z}}}{2i})+iv(\frac{\mathbf{z}+\mathbf{\bar{z}}}{2},\frac{\mathbf{z}-\mathbf{\bar{z}}}{2i})$, and reciprocally.
We shall use the notations: $\d F : =\left(\frac{\partial F}{\partial z_{1}},\cdots,\frac{\partial F}{\partial z_{n}}\right)$,
$\overline{\d}F :=\left(\frac{\partial F}{\partial\overline{z}_{1}},\cdots\frac{\partial F}{\partial\overline{z}_{n}}\right)$, and 
 $\overline{\d F} := \left(\frac{\partial \bar F}{\partial\overline{z}_{1}},\cdots\frac{\partial \bar F}{\partial\overline{z}_{n}}\right)$ is the conjugate of $\d F$.
We then call \textit{mixed polynomial} the following presentation:
\begin{equation}\label{eq:mixed}
 F(\mathbf{z}) = F(\mathbf{z},\mathbf{\bar{z}})=\underset{{\scriptstyle \nu,\mu}}{\sum}c_{\nu,\mu}\mathbf{z}^{\nu}\mathbf{\bar{z}^{\mu}}
\end{equation}

 where $c_{\nu,\mu}\in \bC$,  $\mathbf{z}^{\nu} :=z_{1}^{\nu_{1}}\cdots z_{n}^{\nu_{n}}$ for
$\nu=(\nu_{1},\cdots,\nu_{n})\in\mathbb{N}^{n}$ and  
 $\mathbf{\bar{z}^{\mu}} :=\bar{z}_{1}^{\mu_{1}}\cdots\bar{z}_{n}^{\mu_{n}}$
for $\mu=(\mu_{1},\cdots\mu_{n})\in\mathbb{N}^{n}$. 
 
One may replace ``polynomial'' in the above definition by ``real analytic map germ'' and get the so-called \emph{mixed functions}.

The \emph{singular locus} $\Sing F$ of a mixed function $F$ is by definition the set of critical
points of $F$ as a real-valued map. We have
 $\mathbf{z}\in\Sing F$  (see \cite{Oka1}, \cite{CT}) if and only if there exists $\lambda\in\mathbb{C}$,
$\left|\lambda\right|=1$, such that:
\begin{equation}\label{eq:singmixtes}
\overline{\mathrm{d}F}(\mathbf{z},\overline{\mathbf{z}})=\lambda\overline{\mathrm{d}}F(\mathbf{z},\overline{\mathbf{z}})
\end{equation}

In order to control the limits of tangent spaces to fibres of $F$ we need the following simple result which will play a key role in 
  the proofs:

\begin{lemma}\label{l:M(f)geometric}
Let $F: \bC^n \to \bC$ be a mixed function. The normal vector space to the fibre of $F$ at the point $\bz \in \bC^n$ is generated 
by the following vectors in $\bC^{n}$:
\begin{equation}\label{eq:normal}
 {\mathbf n}_{\mu} :=
\mu\overline{\mathrm{d}F}(\mathbf{z},\overline{\mathbf{z}})+\overline{\mu}\overline{\mathrm{d}}F(\mathbf{z},\overline{\mathbf{z}}),
\end{equation}
for all $\mu\in\bC^*$, $|\mu | =1$.
 \end{lemma}
 
\begin{remark}\label{r:stupiderrorPS} 
As one can see in \eqref{eq:normal}, the set of all normal vectors at some point of a fibre of $f\bar g$ is a real 2-plane but it might be not a complex line, since it is not necessarily invariant under multiplication by $i$.  
\end{remark}
\subsection{Proof of Lemma \ref{l:M(f)geometric}}
Let $F: \bC^n = \mathbb{R}^{2n}\to\mathbb{R}^{2},\, F(\mathbf{z},\overline{\mathbf{z}}) =(\mathrm{Re}F(\mathbf{z},\overline{\mathbf{z}}),\mathrm{Im}F(\mathbf{z},\overline{\mathbf{z}}))$.
Then ${\mathbf n}$ is a normal vector to the fibre of $F$ at $\bz$ iff 
there exist $\alpha,\beta \in\mathbb{R}$
such that ${\mathbf n} =\alpha\mathrm{dRe}F(\mathbf{\bz})+\beta\mathrm{dIm}F(\mathbf{\bz})$.

By displaying this equality we easily get  ${\mathbf n}_k =(\alpha+\beta i)\frac{\partial\bar{F}}{\partial\overline{z}_{k}}+(\alpha-\beta i)\frac{\partial F}{\partial\overline{z}_{k}}$
for every $k\in\{1,\ldots,n\}$.  Our claim follows by taking $\mu=\alpha+\beta i$.

\subsection{Mixed functions of type $f\bar g$}\ 

The real maps $f\bar g: (\bC^n,0)\to (\bC, 0)$, where $f$ and $g$ are holomorphic functions, are  particular mixed functions. A'Campo \cite{AC} had used them to produce examples of  real singularities which are not topologically equivalent to holomorphic ones.  Pichon and Seade studied some topological properties of this class in 2 complex variables \cite{Pi}, \cite{PS}, see also \cite{RSV}.  Ishikawa  \cite{Is1} 
found overtwisted contact structures on $S^3$ which cannot be realised by holomorphic functions. Some topological aspects have been studied in 3 variables by Fernandes de Bobadilla and Menegon Neto \cite{FM}.

\smallskip
 
Our general solution to problem (i) of the Introduction is based on the natural relation of $f\bar g$ to the holomorphic map $(f,g)$, as follows.

\begin{theorem}\label{t:main1}
 Let $f,g : (\bC^n,0)\to (\bC, 0)$, $n>1$,  be some non-constant holomorphic function germs. 
 
 Then
the discriminant $\Disc f\bar g$ of the mixed function germ $f\bar g : (\bC^n,0)\to (\bC, 0)$ 
is either $\{ 0\}$  or a union of finitely many real semi-analytic arcs at the origin.  

Moreover,  $\Disc f\bar g \subset \{0 \}$  if and only if  the discriminant $\Disc (f,g)$ of the map $(f,g)$ contains 
only curve components which are tangent to the coordinate axes.

\end{theorem}

 In the particular case of 2 variables, Pichon and Seade \cite[Theorem 4.4]{PS} had found a solution to problem (i)
 in terms of an embedded resolution of the plane curve  $fg=0$, which also works whenever the source $\bC^{2}$ is replaced by a normal surface. 

\subsection{Proof of Theorem \ref{t:main1}}

We denote $V:= (f\bar g)^{-1}(0) = \{f=0\} \cup \{g=0\}$.  Let us first compare $\Sing f\bar g$ to $\Sing(f,g)$.
\begin{lemma}\label{l:sing1}  
 \
\begin{enumerate}
\item $\Sing f\bar g \m V \subset \Sing (f,g) \m V $.
\item $\Sing f\bar g \cap V = \{ f=g=0 \} \cup \Sing f \cup \Sing g$.
\end{enumerate}
\end{lemma}

 \begin{proof}
In coordinates $(\bz, \bar \bz)$ we have the following characterisation: 
\begin{equation}\label{eq:sing}
 \Sing f\bar g = \{ g \overline{\d f} =  \lambda f \overline{\d g}, \ | \lambda | = 1\}
\end{equation}
 which follows from \eqref{eq:singmixtes}.
On the other hand we have:
\[ \Sing(f,g) = \{ \bz\in \bC^n \mid \rank \langle \d f(\bz), \d g(\bz)\rangle <2\}.\]
 One then derives our statement by a simple computation.
 \end{proof}

 In particular, Lemma \ref{l:sing1} tells  that any point $\bz \not\in \Sing (f,g)\cup V$ is a regular point of the map $f\bar g$.  
Consider a small enough ball $B_{\e}\subset \bC^{n}$ and a disk $D_{\delta}\subset \bC$ with $0< \delta \ll \e$.
It is well-known that one can find complex analytic Whitney stratifications of  the  target (and of the source) such that the map $(f,g) : B_{\e}\cap  (f,g)^{-1}(D_{\delta}) \to D_{\delta}$ is a stratified submersion, see e.g. \cite{GLPW}. This stratification has $D_{\delta} \m \Disc (f,g)$ as its unique 2-dimensional stratum. The other strata of the target are precisely the 1-dimensional strata corresponding to the irreducible components of $\Disc (f,g)$, and the origin is the unique stratum of dimension zero. We recall that all these objects are considered as germs at the origin.

The map $f\bar g$ decomposes as $\bC^n \stackrel{(f,g)}{\rightarrow} \bC^2 \stackrel{u\bar v}{\rightarrow} \bC$.  
We have just seen that the restriction of the map $(f,g)$ to the strata of  $\Sing (f,g)$ is a stratified submersion 
over every corresponding positive dimensional complex stratum of  $D_{\delta}$.  Therefore in order to find the locus 
of non-regular values of the function $f\bar g$ we need to study the restrictions of $u\bar v$ to these strata of $D_{\delta}$.
 First of all, the restriction of $u\bar v$ to the open stratum $D_{\delta} \m \Disc (f,g)$  is a submersion  since $\Sing (u\bar v) = (0,0)$. 
 
 Let us then consider  the restriction of $u\bar v$ to an irreducible component of $\Disc (f,g)$.

\begin{lemma}\label{l:sing}\footnote{The previous version of this criterion was wrong; the problem has been pointed out by Mutsuo Oka.  Its proof contained a mistake in the second part of the computation. This version corrects the computation and gives the adjustment of the  statement and of its proof. Three other statements (Thm. \ref{t:main1}, Thm. \ref{t:mainthom1} and Cor. \ref{c:tubefib}) needed to be reformulated due to this change of criterion, but their proofs are the same. Subsection 4.1 has been completely removed.}
Let $C$ be an irreducible component of the germ $(\Disc (f,g), 0)\subset (\bC^{2}, 0)$.  Then the restriction $(u\bar v)_{|C \m \{ 0\}}$ is a submersion if and only if  $C$ is tangent to one of the coordinate axes without coinciding with it.
\end{lemma}

\begin{proof}
We continue to work with set germs.

In case the irreducible curve germ $C$ is one of the coordinate axes, say $\{ u=0\}$  (and note that $\{ u=0\}\subset \Disc (f,g)$ if and only if   $\Sing f \m \{ g=0\}\not= \emptyset$),  the image  $u\bar v (\{ u=0\})$ is the origin and  the restriction $(u\bar v)_{|C \m \{ 0\}}$ is totally singular.

Let the irreducible curve $C$ be different from the coordinate axes, and let $u=t^p$, $v= a_1 t^q + \mbox{h.o.t.}$ be a Puiseux expansion of  $C$, with $a_1\not= 0$. We may assume without loss of generality that $p\le q$ since the singular locus of $u\bar v$ coincides to that of $v\bar u$, thus 
we may interchange the coordinates $u$ and $v$.

The equality  which defines the singular locus of the function $u(t)\bar v(t)$ is 
$v \overline{\d u} = \lambda u \overline{\d v}$, where $| \lambda | = 1$, and note that $\lambda$ depends on $t$. 
Taking the modulus both sides gives $|v \d u|= |u\d v|$, and thus $| p a_1 t ^{p+q-1} + \mbox{h.o.t.}|
= | q a_1 t ^{p+q-1} + \mbox{h.o.t.}|$. 
Dividing out by $|a_1||t|^{p+q-1}$, and taking the limit for $t\to 0$, one gets the equality
$p=q$. This implies that if $C$ is tangent to the coordinate axes then the restriction $(u\bar v)_{|C \m \{ 0\}}$ is a submersion.

The change with respect to the previous proof starts from now: we show that, reciprocally,   if $p =q$ then the equation $v \overline{\d u} = \lambda u \overline{\d v}$ has solutions, as follows.
\\
If $v= a_1 t^p$ is a one-term expansion, then all points $t$ are solutions of this equation, and with $\lambda \equiv 1$;  consequently the corresponding critical value set contributes with a real  half-line in $\Disc f\bar g$, since $C$ is a line different from the coordinate axes by our assumption $a_1\not= 0$. \\
If  the Puiseux expansion $v= a_1 t^p + a_2 t^{p+j} + \mbox{h.o.t.}$ has at least 2 terms, where $j\ge 1$, then the equivalent equation  $| v \d u | =| u \d v |$ reduces, after dividing out by $p| a_{1}| |t|^{2p-1}$, to an equation of the form
$| 1 + bt^{j}+ \mbox{h.o.t.}| =| 1 + ct^{j}+ \mbox{h.o.t.}|$, with $b,c\not= 0$ and $b\not= c$.  By inverting  one of the holomorphic series, we obtain the equation 
$| 1+ \alpha t^{j} +\mbox{h.o.t.}| =1$ for some $\alpha \not=0$,  which has solutions for all small enough $t$. More precisely, there is at least one\footnote{actually an even number of semi-analytic arcs} real semi-analytic arc $\gamma(s)$  parametrised by $s\in [0, \e[$, for some small enough $\e>0$,
which is solution  of this equation. Then its image  $(u\bar v)(\gamma)$ is included in the discriminant $\Disc f\bar g$.  Since by hypothesis neither $u$ nor $v$ are constant along this arc,  it follows that this image is  not reduced to the point $0$, hence it must be a non-trivial continuous real arc.  This ends the proof of our statement.
  \end{proof}
  
\smallskip
The first assertion of our Theorem \ref{t:main1} follows from the proof of Lemma \ref{l:sing} and from \L ojasiewicz' result that the image by an analytic map of a real semi-analytic arc is a semi-analytic arc.
The second assertion follows since  Lemma \ref{l:sing}  provides the exact characterisation of the condition $\Disc f\bar g \subset \{0\}$ in terms of $\Disc(f,g)$. 


\bigskip

\section{The locus of Thom irregularity for $f\bar g$} \label{proof}

We address here problem (ii) stated in the Introduction:\emph{ in $n>1$ variables, under what conditions $f\bar g$ has a Thom regular stratification.}

Let us first shortly review some existing results in particular settings. 

In case of mixed functions $h : \bC^{n}\to \bC$ which are Newton strongly non-degenerate and non-convenient, Oka  gives in \cite[Theorem 10]{Oka} a sufficient condition for $h$ to be Thom regular, thus have Milnor tube fibration. It also follows that $h$ has an isolated critical value in  this case.

In 2 variables, Pichon and Seade \cite[Proposition 1.4]{PS} claimed that  $f\bar g$ is Thom regular if it has isolated critical value, and if $f$ and $g$ have no common factor. It has been claimed in a subsequent article that this extends to $n$ variables,  which was disproved by a counterexample  \cite{Pa}, \cite{PS2}.  The proof  of \cite{PS} in 2 variables contains itself the same error, namely it tacitly assumes that the normal space at a point of some fibre of $f\bar g$ is generated by a single vector. Our Remark \ref{r:stupiderrorPS} and \eqref{eq:normalvect} below show that the space of  normal vectors at a point of some fibre of $f\bar g$ is a real 2-plane which can be only exceptionally a complex line. The good news is that the statement of \cite[Proposition 1.4]{PS} remains however true since it can be regarded now as a particular case,  in 2 variables,  of  our Theorem \ref{t:mainthom1}, the proof of which is based on totally different principles.
 
In 3 variables,  Fernandes de Bobadilla and Menegon Neto \cite{FM}   give an ``if and only if''  criterion  \cite[Theorem 1.1(i)]{FM}
for the mixed functions of type  $f\bar g$ with  isolated critical value  to have  Milnor tube fibration.  We send to Remark \ref{r:3var} for a comment and an open question which relates this to our forthcoming Theorem \ref {t:homogen}.

\bigskip

We use here a totally different strategy in addressing the Thom regularity, inspired by the work of  Henry, Merle and Sabbah \cite{HMS}.
 Referring to Thom's paper \cite{Th} for the Thom's regularity condition, we introduce the ``Thom irregularity set'':

\[
 NT_{f\bar g} := \left\{ x\in (f\bar g)^{-1}(0) \mid \begin{array}{ll} 
                                            \mbox{ there is no Thom } (a_{f\bar g})\mbox{-stratification of } 
(f\bar g)^{-1}(0)
\\ \mbox{ such that } x \mbox{ is on a positive dimensional stratum. }
                                           \end{array}
\right\}
\]

The set $NT_{(f, g)}$ associated to the map $(f, g)$ instead of $f\bar g$ is defined analogously and has been studied by Henry, Merle and Sabbah \cite{HMS}. Our  result is the following:

\begin{theorem}\label{t:mainthom1}
 Let $f, g : (\bC^n, 0) \to (\bC, 0)$ be any holomorphic germs such that  $\Disc(f, g)$ contains only curves tangent to the coordinate axes. 
  Then
\[
NT_{f\bar g} \subset NT_{(f, g)}.
\]
\end{theorem}

\begin{proof}
Since  the discriminant of $f\bar g$ is the origin only,  the Thom regularity condition is satisfied at every point of $V\m \Sing f\bar g$,
where by $V$ we have denoted the zero locus of $f\bar g$. By Lemma \ref{l:sing1}(b) we have:
\[ \Sing f\bar g  = \{ f=g=0 \} \cup \Sing f \cup \Sing g .\]

The proof of our theorem falls into two main steps.  The first one concerns the Thom regularity along $\Sing f$ and along  $\Sing g$,  outside 
$\{ f=g=0 \}$.
This part is similar to Hironaka's  result about the existence of Thom stratifications in case of holomorphic functions. 
Here, instead of taking Hironaka's viewpoint, we 
inspire ourselves from the proof by  Hamm and L\^{e} \cite{HL} which uses Lojasiewicz 
inequality.

\medskip
\noindent
{\bf Step 1.}  $NT_{f\bar g} \subset \{ f=g=0\}$. 

Let $x_0 \in \Sing f \m \{ g=0\}$. 
There is a small enough neighbourhood of $x_0$ where the following \L ojasiewicz inequality holds:
\begin{equation}\label{eq:loja}
c |f|^\theta \le \| \grad f\|
\end{equation}
for some $\theta \in ]0,1[$ and some $c \in \bR_+$. Moreover, we may assume that in this neighbourhood the real analytic functions $|g|$ and $\| \d g\|$ are bounded and that $|g|> \delta$ for some small enough real positive $\delta$.
 
 In case of $F = f\bar g$, formula \eqref{eq:normal} of Lemma  \ref{l:M(f)geometric} takes the following form:
\begin{equation}\label{eq:normalvect}
 n_{\mu} =
\mu g \overline{\d f}(\bz)+ \overline{\mu} f \overline{\d g} (\bz).
\end{equation}

We consider sequences of points $\bz \to x_0$ in the same neighbourhood, with $f(\bz)\not= 0$.
After  dividing  by $|f|^\theta$ in \eqref{eq:normalvect}, the first term $\mu g \frac{\overline{\d f}}{|f|^\theta}(\bz)$
is bounded away from 0 whereas the second term $\bar \mu \overline{\d g} \frac{f}{|f|^\theta}(\bz)$ tends to 0.
Therefore $\lim_{\bz \to x_0} \frac{1}{|f|^\theta}n_{\mu} = \lim_{\bz \to x_0} \mu g \frac{\overline{\d f}}{|f|^\theta}(\bz)$,
which has the direction of $\lim_{\bz \to x_0} \overline{\d f}(\bz)$, for any $\mu$ with $|\mu | =1$. This is  the same limit of directions as that for the fibres of the holomorphic function $f$ along the sequence of points $\bz\to x_{0}$. Since the zero locus $\{ f=0\}$ admits a Thom (a$_f$)-regular stratification (e.g. by Hironaka's result \cite{Hi}), it  follows that we may  endow the set  $\Sing f \m \{ g=0\}$ with such a Thom regular stratification with complex strata.  Our proof  then shows that this stratification verifies the Thom regularity condition for $f\bar g$.

Switching the roles of $f$ and $g$, the similar conclusion holds for $\Sing f \m \{ g=0\}$.

\medskip
\noindent
{\bf Step 2.} $NT_{f\bar g}\subset \{ f=g=0\} \cap NT_{(f,g)} = NT_{(f,g)}$.

Let $x_{0}\in W \setminus NT_{(f,g)}$, where $W:= \{ f=g=0\}$. We consider sequences of points $\bz \in B_{\e}\m \Sing f\bar g$ tending to $x_{0}$.

By hypothesis the map  $(f,g)$   is  endowed with a Whitney stratification such that some neighbourhood of $x_{0 }\in W$ is the union of finitely many Thom (a$_{(f,g)}$)-regular strata.
Since the fibres of $f\bar g$ are unions of fibres of  the map $(f,g)$, the tangent space at $\bz$ of the fibre of $f\bar g$ contains 
 the tangent space of  the (stratum\footnote{the point $\bz$ might be a singular point of the map $(f,g)$ without being a singular point of the map $f\bar g$, by Lemma \ref{l:sing1}(a), hence we have to take into account our fixed Whitney stratification of the source, and  therefore consider the tangent space to the fibre of the restriction of $(f,g)$ to such a stratum.}  of the) fibre of the map $(f,g)$ to which $\bz$ belongs. 
 
 Taking  the limit $\bz \to x_{0 }\in W$, it follows that 
 the limit of tangent spaces to fibres of $f\bar g$
  contains the limit of the corresponding fibres of $(f,g)$. By construction, the latter contains the tangent space at $x_{0}$ to the stratum of the Thom stratification to which the point $x_{0}$ belongs. This implies that the Thom (a$_{(f,g)}$)-regular stratification of $W$ at $x_{0}$ is also  Thom (a$_{f\bar g}$)-regular.

This ends the proof of our theorem.
  \end{proof}

\begin{remark}
The above proof of Theorem \ref{t:mainthom1} applies to 
  a more general source than $(\bC^{n},0)$,  such as a complex space germ with isolated singularity $(X,0)$.
  \end{remark}

\begin{example}\label{ex:milnortube}  where $NT_{f\bar g} = NT_{(f,g)}$.

Here are two examples, extracted from  \cite[\S 5.3]{ACT}, \cite{Ti-obw}, see also \cite{PS2}, \cite{Oka}.  
In two variables, let  $f\bar g(x,y) = xy\bar x$ (cf \cite{Ti-obw}). In 3 variables, let $f\bar g(x,y,z)= (x+z^2)y\bar x$; this second example is
a deformation of the first and was given by Parusinski \cite{Pa} in response to a question of Tibar in  \cite{Ti-obw}.  It turns out that both examples verify the equality 
$NT_{f\bar g} = NT_{(f,g)}$.  In the first case we have $NT_{f\bar g} = \{ x=0 \} = NT_{(f,g)}$ and in the second
$NT_{f\bar g} = \{ x=z = 0\}= NT_{(f,g)}$. 
\end{example}
We shall see in the next section that for the above examples, despite the  non-existence of Thom stratification,
 the Milnor tube fibration exists since they satisfy our Theorem \ref{t:homogen} below.
 
 \smallskip

\begin{example}\label{ex:sabbah} \label{ex:milnortube2} where $NT_{f\bar g} \not= NT_{(f,g)}$.

Sabbah \cite{Sa1} showed that the map $(f,g) := (x^2 - zy^2,  y)$ is with ``\' eclatement'',  which implies that  it has no Thom regularity at $f=g=0$, thus  $NT_{(f,g)}$ is the $z$-axis.
We have computed $NT_{f\bar g}$ for $f\bar g = (x^2 - zy^2)\bar y$ by using Lemma \ref{l:M(f)geometric}
and found out that $f\bar g$ has the $z$-axis without the origin as a single Thom stratum, thus $NT_{f\bar g}= \{ 0\}$.
 Here are the main steps of the computations.

We claim that all the real normal directions to fibres of $f\bar g$, more precisely all the directions of the generating family of normal vectors:
\[n_{\mu} :=
\mu g \overline{\mathrm{d}f}+\overline{\mu} f \overline{\mathrm{d}}g\]
for $| \mu | =1$,  tend to limits of the form $(*, *, 0) \in \bC^{3}$, where ``$*$'' means some finite value,  and at least one of the two complex numbers $*$ is different from 0. This implies that all the limits of tangents to fibres of $f\bar g$ contain the $z$-axis.

At some point $(x,y,z)\in \bC^{3}$ our generating family of normal vectors to the fibre takes the form:
\[ n_{\mu} = \mu (2\bar x y, -2y\bar y \bar z, - \bar y^2 y) + \bar \mu (0, x^2 - zy^2, 0).
\]
We consider a sequence of points $(x,y,z) \to (0,0, z_{0})$, where  $z_{0}\not= 0$.  Let $y\not=0$ (since in the case $y\equiv 0$ the computation is trivial).  We are interested in the directions of the vectors and not in their size, therefore we may divide by any real number. Therefore, dividing by $|y^2|$, we get:
\begin{equation}\label{eq:n}
 \mu (2\bar x/ \bar y, -2 \bar z, - \bar y) + \bar \mu (0, x^2/|y|^2 - z y^2/|y|^2, 0)
\end{equation}
 
 There are two cases: either $|x/y|$ is bounded or  tends to infinity. In the former,  the limit of  \eqref{eq:n}  is of the form $(*,*, 0)$. In the later, 
the limit of  \eqref{eq:n} is of the form $(0, *, 0)$ where ``$*$''  is finite and different from 0. When we say ``the limit'' we mean all the limits for any $\mu$ with $|\mu| =1$.
 Thus our claim is proved.
\end{example}

\section{Milnor tube fibration with Thom regularity}\label{MilnorwithThom}

 The existence of a Thom $(a_{G})$-stratification of $\Sing G\subset G^{-1}(0)$ implies the existence of the Milnor tube fibration, since the Thom regularity insures the transversality to arbitrarily small spheres
of the fibres of $G$ in some small enough tube, except of the fibre over $0$.
 We derive from Theorem \ref{t:mainthom1} results which enable us to construct classes of
 functions $f\bar g$ with Thom regularity.

\begin{corollary}\label{c:tubefib}\label{t:mainthom2}
Assume that $\Disc(f, g)$ contains only curves tangent to the coordinate axes.
If the map $(f, g)$  is Thom regular, then $f\bar g$ is Thom regular.

In particular, if $(f, g)$ is in addition an ICIS, then $f\bar g$ is Thom regular.

\end{corollary}

\begin{proof}
The hypothesis is just the expression of the condition $\Disc(f\bar g)\subset \{0\}$ provided by Theorem \ref{t:main1}.
The first assertion follows directly from Theorem \ref{t:mainthom1} applied to  this special case.

The inclusion $NT_{(f,g)}\subset \{ f=g=0\} \cap \Sing (f,g)$ follows directly from the definition. In the ICIS 
case we have  $\{ f=g=0\} \cap \Sing (f,g) \subset \{ 0\}$.
 From this and Theorem \ref{t:mainthom1} we get the inclusion
$ NT_{f\bar g} \subset\{ 0\}$, thus the second claim follows too.
\end{proof}

\begin{corollary}[Thom-Sebastiani type]\label{p:thomseb}
Let  $f :(\bC^m, 0) \to (\bC, 0)$ and  $g :(\bC^p, 0) \to (\bC, 0)$ be holomorphic germs in separate variables. Then $f\bar g$ is Thom regular. 
\end{corollary}
\begin{proof} 
Note that $\Disc f\bar g =\{ 0\}$, since the discriminant $\Disc(f,g)$ is the union of the two coordinate axes and we may apply Theorem \ref{t:main1}.
The fibre $(f,g)^{-1}(0,0)$ is a product $Z(f)\times Z(g)$   of the corresponding zero loci of $f$ and of $g$, in separate variables.
All fibres of $(f,g)$ are also products of fibres in separate variables. 
Both holomorphic functions $f$ and $g$ have Thom stratifications of their zero loci and we fix two such stratifications, respectively.
We claim that the product of the two Thom stratifications  is a  Thom (a$_{(f,g)}$)-regular stratification of  $Z(f)\times Z(g)$.
Indeed one considers the regularity conditions for limits of tangents to smooth fibres (over the complement of the axes) and for limits 
of tangents to fibres over each of the two axes. One can check easily that in every case the limits contains the tangent space to the stratum at 
the point of $Z(f)\times Z(g)$, which proves out claim.

Finaly we apply Theorem \ref{t:mainthom1} to conclude that $f\bar g$ is Thom regular. 
\end{proof}


\smallskip
A concrete class of functions $f\bar g$ which are Thom regular, with nonisolated singularity and isolated critical value
is obtained as follows: let $(f,g)$ be any couple of holomorphic functions  in separate variables. 
We have $\Sing (f,g) = \Sing f \times \bC^{p} \bigcup \bC^{m}\times  \Sing g \subset V$,  hence the map germ $(f,g)$ is not an ICIS.   Lemma \ref{l:sing1}(a) implies that $0$ is an isolated singular value of $f\bar g$. Lemma \ref{l:sing1}(b) tells that $\Sing f\bar g$ is nonisolated
as soon as $m>1$ or $p>1$. Then Corollary \ref {p:thomseb} shows that $f\bar g$ is Thom regular.


\section{Milnor tube without Thom regularity}\label{noThom}

In \cite{ACT} it is explained that the Thom (a$_{f\bar g}$)-regularity condition along $V = (f\bar g)^{-1}(0)$ is not necessary for the existence of an open book structure and shows examples which have an $S^{1}$-action.

Namely, one considers weighted homogeneous complex polynomials, a class of singularities which has been deeply studied. They  have a  $\bC^{*}$-action.  As  $\bC^{*}= S^{1}\times \bR_{+}$, this falls into two actions in case of mixed maps. The $S^{1}$-action defines an interesting class of singularities for which several results have been proved  \cite{Oka1}, \cite{CM}, \cite{ACT}, \cite{CT} etc.

 \begin{definition}\label{d:polar} (cf \cite{CM}, \cite{Oka1}) 
 The mixed function $F$ is called \emph{polar} weighted-homogeneous if there are non-zero integers 
$p_{1},\ldots,p_{n}$ and $k$, such that $\gcd(p_{1},\ldots p_{n})=1$ and $\sum_{j=1}^{n}p_{j}(\nu_{j}-\mu_{j})=k$, for any monomial of the expansion \eqref{eq:mixed}.

The corresponding  $S^{1}$-action on $\mathbb{C}^{n}$ is for $\lambda\in S^{1}$: 
\[ 
\lambda\cdot(\mathbf{z},\mathbf{\overline{z}})  =(\lambda^{p_{1}}z_{1},\ldots,\lambda^{p_{n}}z_{n},\lambda^{-p_{1}}\overline{z}_{1},\ldots,\lambda^{-p_{n}}\overline{z}_{n}).
\] 
\end{definition}

   We have the following result:

\begin{theorem}\label{t:homogen} 
Let $F:(\mathbb{C}^{n},0) \rightarrow (\mathbb{C}, 0)$, $n\ge 2$, be a non-constant mixed germ which is polar weighted-homogeneous.
 Then $F$  has  Milnor tube fibration \eqref{eq:tube}.
\end{theorem}

\begin{proof} \sloppy
The proof is almost contained in \cite{ACT}. We emphasize here the main lines.

The existence of  the Milnor tube fibration \eqref{eq:tube} for $F$, where $V:= F^{-1}(0)$,  is implied by the following conditions (see e.g. \cite[Proof of Theorem 2.1]{ACT}):  

\begin{enumerate}
\item $\Sing F \subset V$, and
 \item  $\overline{M(F)\m V} \cap  V = \{ 0\}$,
 \end{enumerate}
where $M(F) := \{ x\in U \mid \rho \not\pitchfork_x F\}$,  $\rho : U \to \bR_{\ge 0}$ is the Euclidean distance and 
$U\subset \bC$  is an open neighborhood of  the origin.
The subset $M(F)$ is the set of non-transversality between $\rho$ and $F$, is called the set of \textit{$\rho$-nonregular points}
 and is also known as \textit{Milnor set}\footnote{the ``$\rho$-regularity'' has been introduced in \cite{Ti-cras} and the ``Milnor set'' appears in \cite{NZ}.}.
 
Let us assume that conditions (a) and (b) are fulfilled and that the image of $F$ contains a neighbourhood of the origin of $\bC$. 

Since condition (b)  is equivalent to the transversality of the fibres of $F$ to the boundary $S^{2n-1}_\e = \partial \bar B^{2n}_{\e}$ in some small enough neighborhood of $V$ depending on $\e$,  it follows that the restriction of $F$ to the tube,  as given in  Definition \ref{d:tube} is proper.
Condition (a) implies that $F$ is a submersion outside $V$, i.e. that $0$ is an isolated critical value of $F$.

Then we may apply  Ehresmann's theorem \cite{Eh, Wo} for manifolds with boundary to conclude to the existence of the locally trivial fibration (\ref{eq:tube}).

\medskip

Let us now prove  conditions (a) and (b) for our $F$.

We first prove that $\im F$ contains a small enough disk at $0 \in \bC$. 
 Since  $F\not\equiv 0$ by hypothesis,  and since $\im F$ is a semi-analytic set germ, by the Curve Selection Lemma, $\im F$ contains a curve $\gamma$ which intersects the circles $S^1_\eta \subset \bC$ for any small enough radius $\eta>0$. 
Take now some $a\in S^1_\eta \cap \gamma$ and $z\in F^{-1}(a)$. Since $F(\lambda \cdot (\mathbf{z}, \mathbf{\overline{z}})) = \lambda^k F(\mathbf{z}, \mathbf{\overline{z}})$, we have $\lambda^k a \in \im F$ for any $\lambda \in S^1$. This shows that $\im F$ actually contains a disk $D^2_{\eta_0}$, for some small enough $\eta_0>0$.

The germ at $0$ of the set of critical values of $F$ is a semi-analytic set of dimension $\le 1$. Take its complement in $\im F$, which is a 2 dimensional semi-analytic set germ at $0$. Applying the above reasoning yields that all values different from $0$ are regular, hence their inverse images are manifolds of dimension $2n-2\ge 1$. This proves (a).

Take now the restriction of $F$ to some small enough sphere $S_\e^{m-1}$. Its image must contain a non-constant curve germ at $0$. Since the $S^1$-action preserves the sphere, by the same reasoning as above, the image of $F_{|S_\e^{m-1}}$ contains some disk $D^2_{\eta_0}$. The regular values of $F_{|S_\e^{m-1}}$ are a dense semi-analytic set and if $a$ is a regular value then $\lambda^k a$ is regular too, for any $\lambda \in S^1$. Hence all values of the pointed disk $D^2_{\eta_0}\m \{ 0\}$ are regular.   Since this holds for any small enough $\e >0$,
we have proved (b).
\end{proof}


\subsection{Classes of Milnor tube fibration without Thom regularity}\label{ss:noThom} \

The class of  mixed
 functions $F_{k} := \bar x y(x+z^{k})$, $k\ge 2$, is with isolated critical value, nonisolated singularity and has  Milnor tube fibration but not Thom regularity. This extends Example \ref{ex:milnortube}. 
 
 Indeed, remark first that  $f\bar g$ is polar homogeneous (cf Definition \ref{d:polar}), hence  it has isolated critical value and Milnor tube fibration due to Theorem \ref{t:homogen}. The singular set contains the $y$-axis and the $z$-axis.
 
 Let us see that $F_{k}$ does not have Thom regularity at $V$.  
 The family of normal vectors at some point $(x,y,z)$ is:
 \[  n_{\mu} = \mu (x \bar y,  x \bar x + x \bar z^{k},  kx\bar z^{k-1}\bar y) + \bar \mu ( y(x+z^{k}), 0 , 0)
\]
for any $\mu$ with $|\mu| =1$. We check the Thom regularity at some point $(0, y_{0}, 0)$ of the $y$-axis.
Let us point out a general fact  which simplifies the computations:  the $S^{1}$-action on  $F_{k}$ implies that the Thom strata of  $V$ are $S^{1}$-invariant. Thus the  $y$-axis minus the origin must be a single stratum. 

Therefore, without loss of generality we may choose $y_{0}\in \bR\m \{ 0\}$. We also choose 
a particular sequence of points of the form $(x,y,0)$ tending to $(0, y_{0}, 0)$.  Then:
 \[  n_{\mu} = \mu (x \bar y,  x \bar x ,  0) + \bar \mu ( yx, 0 , 0) 
\]
and now choose $\mu = i$. We get:
 \[  n_{\mu} = (i x( \bar y - y),  i x \bar x ,  0)  = (0,  i x \bar x ,  0).
\]
The limit is of the form $(0, * , 0)$, but if  there were Thom regularity at the $y$-axis then all the limits should be of the form $(*, 0, *)$. This gives a contradiction.


\medskip
\begin{remark}\label{r:3var}
 Fernandes de Bobadilla and Menegon Neto  state in \cite[page 154]{FM} that the example
 $f\bar g(x,y,z)= (x+z^2)y\bar x$ does not satisfy their criterion \cite[Theorem 1.1(i)]{FM} for the 3-variables functions of type  $f\bar g$   to have  Milnor tube fibration.  Nevertheless  our Theorem \ref {t:homogen} shows that this example  has Milnor tube fibration, and the above computation shows that it is not Thom regular. Definitely,  there  is a contradiction here.  
It may come from the
  different definition of \emph{Milnor tube} that the authors use  in \cite{FM}, with polydisk neighbourhoods instead of the usual ball neighbourhoods like in our Definition \ref{d:tube}.   We leave open this interesting problem.

\end{remark}


\end{document}